\documentclass[preprint,12pt]{elsarticle}

\usepackage{amssymb}
\usepackage{amsthm}
\usepackage{amsmath}
\usepackage{amsfonts}

\newtheorem{theorem}{Theorem}[section]

\newtheorem{lemma}[theorem]{Lemma}

\newtheorem{proposition}[theorem]{Proposition}

\theoremstyle{definition}

\newtheorem{remark}[theorem]{Remark}

\numberwithin{equation}{section}

\def\Plambda{(1.1)_{\lambda}}
\def\Pext{(1.1)_{\lambda^\star}}

\begin{document}

\begin{frontmatter}

\title{$W^{1,q}$ estimates for the extremal 
solution of reaction-diffusion problems}
\author{Manel Sanch\'on}
\ead{msanchon@maia.ub.es}
\address{Departament de Matem\`atica Aplicada i An\`alisi,
Universitat de Barcelona, Gran Via {\rm 585}, {\rm 08007} Barcelona, Spain}

\begin{abstract}
We establish a new $W^{1,2\frac{n-1}{n-2}}$ estimate for the extremal solution 
of $-\Delta u=\lambda f(u)$ in a smooth bounded domain $\Omega$ of $\mathbb{R}^n$, 
which is convex, for arbitrary positive and increasing nonlinearities $f\in C^1(\mathbb{R})$ 
satisfying $\lim_{t\rightarrow +\infty}f(t)/t=+\infty$.  
\end{abstract}


\end{frontmatter}


\section{Introduction}
Let $\Omega$ be a smooth bounded domain of $\mathbb{R}^n$ and consider the 
reaction-diffusion problem
\stepcounter{equation}
$$
\left\{
\begin{array}{rcll}
-\Delta u&=&\lambda f(u)&\textrm{in }\Omega,\\
u&=&0&\textrm{on }\partial \Omega,
\end{array}
\right. \eqno{(1.1)_{\lambda}}
$$
where $\lambda$ is a positive parameter and $f$ is a positive and 
increasing $C^1$ function satisfying 
\begin{equation}\label{superlinear}
\lim_{t\rightarrow+\infty}\frac{f(t)}{t}=+\infty.
\end{equation}

Crandall and Rabinowitz~\cite{CR75} proved, using the Implicit 
Function Theorem, the existence of an extremal parameter $\lambda^\star
\in(0+\infty)$ such that problem $\Plambda$ admits a classical 
minimal solution $u_\lambda$ for all $\lambda\in (0,\lambda^\star)$. 
Here, minimal means that it is smaller than any other nonnegative solution. 
Moreover, the least eigenvalue of the linearized operator at $u_\lambda$, 
$-\Delta -\lambda f'(u_\lambda)$, is positive for all $\lambda\in(0,\lambda^\star)$. 
Alternatively, this can be reached by using an iteration argument to obtain that 
$u_\lambda$ is an absolute minimizer of the associated energy 
functional 
\begin{equation}\label{functional}
J(u_\lambda):=\int_\Omega|\nabla u_\lambda|^2-\lambda F(u_\lambda)\,dx,
\end{equation}
in the convex set $\{w\in H^1_0(\Omega):0\leq w\leq u_\lambda\}$, where $F'=f$. 
In particular, $u_\lambda$ will be semi-stable in the sense that the second 
variation of energy at $u_\lambda$ is nonnegative definite:
\begin{equation}\label{semi1}
Q_{u_\lambda}(\varphi):=\int_\Omega |\nabla \varphi|^2 -\lambda f'(u_\lambda)\varphi^2\,dx\geq 0
\quad\textrm{for all }\varphi\in  C^1_0(\Omega).
\end{equation}
Brezis \textit{et al.}~\cite{BCMR96} proved that there is no weak solution for 
$\lambda>\lambda^\star$, while the increasing limit 
$$
u^\star:=\lim_{\lambda\uparrow\lambda^\star}u_\lambda
$$
is a weak solution of the extremal problem $\Pext$, \textit{i.e.}, $u^\star\in L^1(\Omega)$, 
\linebreak $f(u^\star)\,{\rm dist}(\cdot,\partial\Omega)\in L^1(\Omega)$, and 
$$
\int_\Omega u^\star (-\Delta\varphi)\,dx=\lambda\int_\Omega f(u^\star)\varphi\,dx
\quad \textrm{for all }\varphi\in C^2_0(\overline{\Omega}).
$$
This solution is known as \textit{extremal solution} of the extremal 
problem $\Pext$. 

The study of the regularity of the extremal solution started to growth after 
Brezis and V\'azquez raised some open problems in \cite{BV97}. In this direction, 
Nedev~\cite{Nedev01} proved, in an unpublished preprint, that $u^\star\in H^1_0(\Omega)$
for every positive and increasing nonlinearity $f$ satisfying \eqref{superlinear} 
when the domain is convex (see also Theorem~2.9 in \cite{CS11}). 
The proof uses the Poho${\rm\check{z}}$aev identity and the fact that $u_\lambda$ is 
an absolute minimizer of the functional $J$, defined in \eqref{functional}, on the compact set 
$\{w\in H^1_0(\Omega):0\leq w\leq u_\lambda\}$, and hence, $J(u_\lambda)\leq J(0)=0$.

Our main result establishes that $u^\star\in W^{1,2\frac{n-1}{n-2}}_0(\Omega)$ 
for any convex domain $\Omega$ and any nonlinearity $f$satisfying the above assumptions. 
In particular, $u^\star\in H^1_0(\Omega)$. We prove it using a geometric Sobolev inequality 
on the graph of minimal solutions $u_\lambda$.
\begin{theorem}\label{theorem}
Let $\Omega$ be a smooth bounded domain of $\mathbb{R}^n$ with $n\geq 3$ and $f$ 
a positive and increasing $C^1$ function satisfying \eqref{superlinear}. Let 
$u_\lambda\in C^2_0(\overline{\Omega})$ be the minimal solution of 
$\Plambda$ for $\lambda\in(0,\lambda^\star)$ and
$$
I(t):=\int_{\{u_\lambda\geq t\}}(1+|\nabla u_\lambda|^2)^{\frac{n-1}{n-2}}\,dx,
\quad t\in(0,\|u_\lambda\|_{L^\infty(\Omega)}).
$$
 
There exists a positive constant $C$ depending only on $n$ such that the 
following inequality holds
\begin{equation}\label{estimate}
\begin{array}{l}
\displaystyle C I(t)^{2\frac{n-1}{n}}
\leq 
\frac{1}{t^2}\left(\int_{\{u_\lambda \leq t\}} (1+|\nabla u_\lambda|^2)|\nabla u_\lambda|^2 \,dx\right)
I(t)
\\
\displaystyle \hspace{3.5cm}
+\left(\int_{\{u_\lambda=t\}}(1+|\nabla u_\lambda|^2)^{\frac{1}{2}\frac{n-1}{n-2}}\,dS\right)^2
\end{array}
\end{equation}
for a.e. $t\in(0,\|u_\lambda\|_{L^\infty(\Omega)})$.

If in addition $\Omega$ is convex then 
the extremal solution $u^\star\in W^{1,2\frac{n-1}{n-2}}_0(\Omega)$. 
\end{theorem}

In the last decade several authors studied the regularity of the extremal 
solution (see the monograph by Dupaigne~\cite{Dupaigne} and references therein). However, 
there are few results for general reaction terms $f$
(\textit{i.e.}, positive and increasing nonlinearities satisfying 
\eqref{superlinear}). 
Cabr\'e~\cite{Cabre10} established that $u^\star\in 
L^\infty(\Omega)$ when $n\leq4$ and the domain is convex. More recently, 
Cabr\'e and the author \cite{CS11} proved for $n\geq 5$ that there exists a constant $C$ 
depending only on $n$ such that
$$
\left(\int_{\{u_\lambda>t\}} \Big(u_\lambda-t\Big)^{\frac{2n}{n-4}}\ dx\right)^\frac{n-4}{2n} 
\leq 
\frac{C}{t}\left(\int_{\{u_\lambda\leq t\}} |\nabla u_\lambda|^4 \ dx\right)^{1/2}
$$
for all $t\in (0,\|u_\lambda\|_{L^\infty(\Omega)})$. 
As a consequence, it is proved that the extremal solution $u^\star$ belongs to 
$L^\frac{2n}{n-4}(\Omega)$ when the domain is convex and the dimension 
$n\geq 5$. 
The first step in the proof of both results is to take $\varphi=|\nabla u_\lambda|\eta$ 
as a test function in the semistability condition \eqref{semi1} and use the 
following geometric identity
\begin{equation}\label{SZ:identity}
\left(\Delta |\nabla u_\lambda|+\lambda f'(u_\lambda)|\nabla u_\lambda|\right)|\nabla u_\lambda|
=\bar{A}^2|\nabla u_\lambda|^2+|\nabla_{\bar{T}}|\nabla u_\lambda||^2 
\end{equation}
in $\{x\in\Omega:|\nabla u_\lambda|> 0\}$, where $\bar{A}^2(x)$ denotes the second fundamental form at $x$ of 
the $(n-1)$-dimensional hypersurface $\{y\in\Omega: |u_\lambda(y)|=|u_\lambda(x)|\}$ 
and $\nabla_{\bar{T}}$ is the tangential gradient with respect 
to this level set. Sternberg and Zumbrun~\cite{SZ1,SZ2} made this choice 
to obtain
$$
Q_{u_\lambda}(|\nabla u_\lambda|\eta)
=
\int_{\Omega\cap\{|\nabla u_\lambda|>0\} }|\nabla u_\lambda|^2|\nabla \eta|^2
    -\left(\bar{A}^2|\nabla u_\lambda|^2+|\nabla_{\bar{T}}|\nabla u_\lambda||^2\right)\eta^2\,dx
$$
for every Lipschitz function $\eta$ in $\overline{\Omega}$ such that 
$\eta|_{\partial\Omega}\equiv 0$, where $Q_{u_\lambda}$ is the quadratic form defined 
in \eqref{semi1}. 
The second step in the proof is to choose an appropriate function $\eta=\eta(u)$ 
and use the coarea formula and a Sobolev inequality on the $(n-1)$-dimensional 
hypersurface $\{y\in\Omega:u_\lambda(y)=u_\lambda(x)\}$.

The first ingredient in the proof of Theorem~\ref{theorem} is the following identity,
analogue to \eqref{SZ:identity}, involving the second fundamental form of ${\rm Graph}(u_\lambda)$.

\begin{proposition}\label{prop1}
Let $u\in C^3_0(\overline{\Omega})$ be a positive function and 
$v(x,x_{n+1}):=u(x)-x_{n+1}$ for all $(x,x_{n+1})\in\Omega\times\mathbb{R}$. Let 
$\nu=-\frac{\nabla v}{|\nabla v|}\in\mathbb{R}^{n+1}$ be the 
unit normal vector to ${\rm Graph}(u)$, $A^2$ the second fundamental 
form of ${\rm Graph}(u)$, and 
$\nabla _T\varphi :=\nabla\varphi -(\nu\cdot\nabla\varphi)\nu$ for every 
$\varphi\in C^1(\mathbb{R}^{n+1})$.
The following identity holds
\begin{equation}\label{key}
\left(\Delta |\nabla v|+\nu\cdot\nabla \Delta v\right)|\nabla v|
=
A^2|\nabla v|^2+|\nabla_T|\nabla v||^2\quad \textrm{in }\Omega.
\end{equation}
In particular, if $u\in C^2(\overline{\Omega})$ is a solution of $\Plambda$ and 
$f\in C^1(\mathbb{R})$ then 
\begin{equation}\label{identity}
\left(\Delta |\nabla v|+\lambda f'(u)|\nabla v|\right)|\nabla v|
=\lambda f'(u)+A^2|\nabla v|^2+|\nabla_T|\nabla v||^2\quad \textrm{in }\Omega.
\end{equation}
\end{proposition}

\begin{remark}
(i) Let $u\in C^2(\overline{\Omega})$ be a solution of $\Plambda$. Note that 
$$
\Delta v=\sum_{i=1}^{n+1} v_{ii}=\sum_{i=1}^{n} u_{ii}=\Delta u
$$
and 
$$
\nabla \Delta v=(\nabla \Delta u,0)=(-\lambda f'(u)\nabla u,0)\in\mathbb{R}^{n+1}.
$$

(ii) Farina, Sciunzi, and Valdinoci \cite{FSV08} and Cesaroni, Novaga, 
and Valdinoci \cite{CNV} recently used identity \eqref{SZ:identity} to 
obtain one-dimensional symmetry of solutions to some reaction-diffusion 
equations. In this sense identity \eqref{identity} could be useful by 
itself.
\end{remark}

The main novelty in the proof of Theorem~\ref{theorem} is that we use a 
Sobolev inequality on the $n$-dimensional hypersurface 
$$
{\rm Graph}(u_\lambda)=\{(x,x_{n+1})\in \Omega\times\mathbb{R}: x_{n+1}=u_\lambda(x)\}
\subset\mathbb{R}^{n+1},
$$
instead on the level sets $\{y\in \Omega:u_\lambda(y)=u_\lambda(x)\}$ of $u_\lambda$ as in \cite{Cabre10,CS11}, 
and the geometric identity \eqref{identity}. More precisely, 
define $v_\lambda(x,x_{n+1}):=u_\lambda(x)-x_{n+1}$ for every $\lambda\in(0,\lambda^\star)$. Taking 
$\varphi=|\nabla v_\lambda|\eta$ in the semistability condition \eqref{semi1} and using 
identity \eqref{identity}, we obtain 
\begin{equation}\label{semi2}
\int_\Omega \Big(\lambda f'(u_\lambda)+A^2|\nabla v_\lambda|^2+|\nabla_T|\nabla v_\lambda||^2\Big)\eta^2\,dx
\leq
\int_\Omega |\nabla \eta|^2|\nabla v_\lambda|^2\,dx
\end{equation}
for every Lipschitz function $\eta$ in $\overline{\Omega}$ such that 
$\eta|_{\partial\Omega}\equiv 0$. 
Choosing $\eta=\min\{u_\lambda,t\}$ as a test function in \eqref{semi2} 
and using a geometric Sobolev inequality on the $n$-dimensional hypersurface 
$\{(x,x_{n+1})\in{\rm Graph}(u_\lambda):x_{n+1}\geq t\}$ 
(see Theorem~\ref{thm:Sobolev} below) we prove inequality \eqref{estimate} in Theorem~\ref{theorem}. 
The $W^{1,2\frac{n-1}{n-2}}$-estimate for the extremal solution 
follows from \eqref{estimate} and the convexity of the domain.

The paper is organized as follows. In section~\ref{section2} we recall a 
Sobolev inequality on $n$-dimensional hypersurfaces with boundary and we 
prove the geometric identities established in Proposition~\ref{prop1}. In 
section~\ref{section3} we prove Theorem~\ref{theorem}.
\section{Geometric indentities and inequalities. Proof of Proposition~\ref{prop1}}\label{section2}

The first ingredient in the proof of Theorem~\ref{theorem} is the 
following Sobolev inequality on $n$-dimensional hypersurfaces (see section 
28.5.3 in \cite{Burago}): \textit{Let  $M\subset  \mathbb{R}^{n+1}$ be a $C^2$ 
immersed $n$-dimensional compact hypersurface with $n\geq 2$.
There exists a constant $C(n)$ depending only on the dimension $n$ such that, 
for every $\phi\in C^1(M)$ it holds
\begin{equation}\label{MSsob1}
C(n)\left(\int_{M}|\phi|^\frac{n}{n-1}\,dV\right)^{\frac{n-1}{n}}
\leq 
\int_{M}(|H\phi|+|\nabla\phi|)\,dV+\int_{\partial M}|\phi|\,dS,
\end{equation}
where $H$ is the mean curvature of $M$.}

Let $p^\star:=np/(n-p)$ be the critical Sobolev exponent. 
Replacing $\phi$ by $\phi^\alpha$ in \eqref{MSsob1}, with 
$\alpha=2^\star/1^\star=2(n-1)/(n-2)$, and using H\"older and 
Minkowski inequalities it is standard to obtain the following result.

\begin{theorem}[\cite{Burago}]\label{thm:Sobolev}
Let  $M\subset  \mathbb{R}^{n+1}$ be a $C^2$ immersed $n$-dimensional 
compact hypersurface with $n\geq 3$.
There exists a constant $C = C(n)$ depending only on the dimension $n$ such that, 
for every $\phi\in C^1(M)$ it holds
\begin{equation}\label{MSsob}
\begin{array}{l}
\displaystyle C \left(\int_{M}|\phi|^{2^\star}\,dV\right)^{2\frac{n-1}{n}}
\leq 
\left(\int_{M}|\phi|^{2^\star}\,dV\right)
\left(\int_{M}(|H\phi|^2+|\nabla\phi|^2)\,dV\right)\\
\displaystyle \hspace{6cm}
+\left(\int_{\partial M}|\phi|^{2\frac{n-1}{n-2}}\,dS\right)^2,
\end{array}
\end{equation}
where $H$ is the mean curvature of $M$ and $2^\star = 2n/(n - 2)$.
\end{theorem}

The second ingredient is identity \eqref{identity} in Proposition~\ref{prop1}. Before to 
prove it let us introduce some notation. 
Let $\Omega$ be a smooth bounded domain of $\mathbb{R}^n$, $v\in C^2(\Omega\times\mathbb{R})$, and
$$
\nu(x,x_{n+1})=-\frac{\nabla v}{|\nabla v|}(x,x_{n+1})
$$
the unit normal vector to the level set of $v$ passing throughout 
$(x,x_{n+1})\in\{|\nabla v|\neq 0\}$. Recall that the 
eigenvalues of $\nu$ are the $n$ principal curvatures $\kappa_1,\cdots,
\kappa_{n}$ of the level sets of $v$ and zero. In particular, 
the second fundamental form $A^2:=\kappa_1^2+\cdots+\kappa_{n}^2$ 
of the level sets of $v$ is given by $A^2=\nu^i_j\nu^j_i$, where as usual 
Einstein summation convention is used. We denote the gradient along the 
level sets of $v$ by $\nabla_T$, \textit{i.e.},
$$
\nabla_T\phi=\nabla\phi-(\nabla\phi\cdot\nu)\nu\quad\textrm{for any }
\phi\in C^1(\mathbb{R}^{n+1}).
$$

Let us prove the identities established in Proposition~\ref{prop1}.
\begin{proof}[Proof of Proposition~{\rm \ref{prop1}}]
Let $u\in C^3_0(\overline{\Omega})$ be a positive function and define 
$v(x,x_{n+1})=u(x)-x_{n+1}$ for all $x\in\Omega$. 

We claim that $\nabla_T{\rm log}|\nabla v|=(D\nu)\nu$. Indeed, noting that
$$
\begin{array}{lll}
\displaystyle -\frac{v_{ij}}{|\nabla v|}
&=&\displaystyle 
\frac{\left(\nu^i|\nabla v|\right)_j}{|\nabla v|}
=
\nu^i\nabla^j{\rm log}|\nabla v|+\nu^i_j
\\
&=&\displaystyle 
\nu^i\nabla^j_T{\rm log }|\nabla v|+(\nabla {\rm log }|\nabla v|\cdot\nu)\nu^j\nu^i+\nu^i_j
\end{array}
$$
and $v_{ij}=v_{ji}$ for all $i,j=1,\cdots,n+1$, we obtain 
$$
\nu^i_j=\nu^j_i+\nu^j\nabla_T^i{\rm log }|\nabla v|-\nu^i\nabla_T^j{\rm log }|\nabla v|
\quad \textrm{for all }i,j=1,\cdots,n+1.
$$
We prove the claim multiplying the previous equality by $\nu^j$ and noting 
that $\nu^i_j\nu^i=0$ for every $j=1,\cdots,n+1$ and $\nabla_T{\rm log }|\nabla v|\cdot\nu=0$.

Now, using $\nu^i_j\nu^j_i=A^2$ and $\nabla_T^j{\rm log}|\nabla v|=\nu^j_i\nu^i$, 
we compute
$$
\begin {array}{lll}
\Delta |\nabla v|&=&-\displaystyle (v_{ij}\nu^j)_{i}
=-\nu\cdot\nabla \Delta v-v_{ij}\nu^j_i
\\
&=&\displaystyle-\nu\cdot\nabla \Delta v+\left(|\nabla v|\nu^i\right)_j\nu^j_i
\\
&=&\displaystyle 
-\nu\cdot\nabla \Delta v+|\nabla v|\nu^i_j\nu^j_i+|\nabla v|_j\nabla_T^j{\rm log}|\nabla v|
\\
&=&\displaystyle 
-\nu\cdot\nabla \Delta v+(A^2+|\nabla_T{\rm log}|\nabla v||^2)|\nabla v|
\end{array}
$$
to obtain identity \eqref{key}.

If $u\in C^2(\overline{\Omega})$ is a solution of $\Plambda$ and 
$f\in C^1(\mathbb{R})$, then by standard regularity results for 
uniformly elliptic equations one has $u\in C^3(\overline{\Omega})$. 
From \eqref{key} and noting that 
$$
\nabla\Delta v=(-\lambda f'(u)\nabla u,0)\quad 
\textrm{and}\quad 
\nu=\frac{1}{|\nabla v|}(-\nabla u,1),
$$ 
we obtain
$$
\Delta |\nabla v|
=
-\lambda f'(u)\frac{|\nabla u|^2}{|\nabla v|}+(A^2+|\nabla_T{\rm log}|\nabla v||^2)|\nabla v|
$$
proving the proposition.
\end{proof}

\section{Proof of Theorem~\ref{theorem}}\label{section3}

Let $u_\lambda$ be the minimal solution of $\Plambda$ for $\lambda\in(0,\lambda^\star)$. 
Choosing $\varphi=\sqrt{1+|\nabla u_\lambda|^2}\,\eta$ as a test function in the 
semistability condition \eqref{semi1} and using Proposition~\ref{prop1}, 
we first obtain~\eqref{semi2}.
\begin{lemma}\label{lemma3:1}
Assume that $\Omega$ is a smooth bounded domain of $\mathbb{R}^n$ and $f$ 
a positive and increasing $C^1$ function satisfying \eqref{superlinear}. Let $u_\lambda$
be the minimal solution of $\Plambda$ and $v_\lambda(x,x_{n+1}):=u_\lambda(x)-x_{n+1}$
for $\lambda\in(0,\lambda^\star)$. The following inequality holds 
\begin{equation}\label{quadratic_form}
\int_\Omega\left(\lambda f'(u_\lambda)+A^2|\nabla v_\lambda|^2+|\nabla_T|\nabla v_\lambda||^2\right)\eta^2\,dx
\leq 
\int_\Omega |\nabla v_\lambda|^2|\nabla\eta|^2 \,dx
\end{equation}
for every Lipschitz function $\eta$ in $\overline{\Omega}$ with 
$\eta|_{\partial\Omega}\equiv0$, where $A^2$ and $\nabla _T$ are 
as in Proposition~{\rm\ref{prop1}}.
\end{lemma}
\begin{proof}
In order to improve the notation, let us denote $u_\lambda=u$ and $v_\lambda=v$ 
for $\lambda\in(0,\lambda^\star)$.
Choosing $\varphi=|\nabla v|\eta$ as a test function in \eqref{semi1} and 
integrating by parts we get 
$$
\begin{array}{l}
0\leq Q_u(|\nabla v|\eta)\\
\hspace{0.2cm}=
\displaystyle 
\int_\Omega |\nabla v|^2|\nabla\eta|^2
   +|\nabla v|\nabla |\nabla v|\cdot\nabla\eta^2
   +|\nabla |\nabla v||^2\eta^2
   -\lambda f'(u)|\nabla v|^2\eta^2\,dx
\\
\hspace{0.2cm}=\displaystyle 
\int_\Omega |\nabla v|^2|\nabla\eta|^2
   -({\rm div}(|\nabla v|\nabla |\nabla v|)
   -|\nabla |\nabla v||^2
   +\lambda f'(u)|\nabla v|^2)\eta^2\,dx
\\
\hspace{0.2cm}=\displaystyle 
\int_\Omega |\nabla v|^2|\nabla\eta|^2
   -\left(|\nabla v|\Delta |\nabla v|
   +\lambda f'(u)|\nabla v|^2\right)\eta^2\,dx.
\end{array}
$$
Inequality \eqref{quadratic_form} follows directly from identity \eqref{identity}.
\end{proof}

Finally, using Lemma~\ref{lemma3:1} and the geometric Sobolev inequality 
established in Theorem~\ref{thm:Sobolev} we prove Theorem~\ref{theorem}.
\begin{proof}[Proof of Theorem~{\rm\ref{theorem}}]
Let $u_\lambda\in C^2_0(\overline{\Omega})$ be the minimal solution of $\Plambda$ 
for $\lambda\in(0,\lambda^\star)$ and $t\in(0,\|u_\lambda\|_{L^\infty(\Omega)})$. 
Define $v_\lambda(x,x_{n+1})=u_\lambda(x)-x_{n+1}$. 
Let $M_t:=\{(x,x_{n+1})\in{\rm Graph}(u_\lambda):x_{n+1}\geq t\}$ and 
$dV=\sqrt{1+|\nabla u_\lambda|^2}\,dx$ its element of volume.

We start by proving inequality \eqref{estimate}. 
On the one hand, taking $\eta=\min\{u_\lambda,t\}$ as a test function in 
\eqref{quadratic_form}, using that $f$ is an increasing function,
and $H^2=(\kappa_1+\cdots+\kappa_n)^2\leq nA^2=n(\kappa_1^2+\cdots+\kappa_{n}^2)$, 
we obtain
\begin{eqnarray}\label{cris}
\displaystyle 
\int_{M_t}\hspace{-0.2cm}\left(H^2|\nabla v_\lambda|+|\nabla_T|\nabla v_\lambda|^\frac{1}{2}|^2\right)\,dV
&\hspace{-0.3cm}\leq&\displaystyle \hspace{-0.3cm}
\int_{\{u_\lambda\geq t\}}\hspace{-0.1cm}\left(nA^2|\nabla v_\lambda|^2+\frac{1}{4}|\nabla_T|\nabla v_\lambda||^2\right)\,dx
\nonumber
\\
&\hspace{-0.3cm}\leq&\displaystyle \hspace{-0.3cm}
\frac{n}{t^2}\int_{\{u_\lambda\leq t\}} |\nabla v_\lambda|^2|\nabla u_\lambda|^2 \,dx
\end{eqnarray}
for all $t\in(0,\|u_\lambda\|_{L^\infty(\Omega)})$.

Therefore, applying Theorem~\ref{thm:Sobolev} with 
$M=M_t$ and $\phi=|\nabla v_\lambda|^{1/2}$, we obtain
\begin{eqnarray}\label{last:estimate}
\displaystyle C \left(\int_{M_t}\hspace{-0.1cm}|\nabla v_\lambda|^{\frac{n}{n-2}}\,dV\right)^{2\frac{n-1}{n}}
\hspace{-0.5cm}&\leq& \hspace{-0.2cm}
\frac{n}{t^2}\left(\int_{\{u_\lambda\leq t\}} \hspace{-0.2cm}|\nabla v_\lambda|^2|\nabla u_\lambda|^2 \,dx\right)
\left(\int_{M_t}\hspace{-0.1cm}|\nabla v_\lambda|^{\frac{n}{n-2}}\,dV\right)\nonumber
\\
\displaystyle 
&&+\left(\int_{\partial M_t}|\nabla v_\lambda|^{\frac{n-1}{n-2}}\,dS\right)^2,
\end{eqnarray}
where $C$ is a constant depending only on $n$. This is inequality \eqref{estimate}.

Assume in addition that $\Omega$ is convex. To prove that the extremal solution $u^\star$ 
belongs to $W^{1,2\frac{n-1}{n-2}}_0(\Omega)$ we only need to bound the integrals on 
$\{u_\lambda\leq t\}$ and $\partial M_t$, for some $t$, by a constant independent of $\lambda$ and then let 
$\lambda$ tend to $\lambda^\star$. The same argument was done in the proof of 
Theorem~2.7 \cite{CS11}. However, for convinience to the reader we sketch the proof.

Since $\Omega$ is convex, there exist positive constants $\varepsilon$ 
and $\gamma$ independent of $\lambda$ such that
\begin{equation}\label{newnew}
\Vert u_\lambda\Vert_{L^\infty(\Omega_\varepsilon)}
\leq 
\frac{1}{\gamma} \Vert u^\star\Vert_{L^1 (\Omega)}
\qquad\textrm{for all }\lambda<\lambda^\star,
\end{equation}
where  $\Omega_\varepsilon:=\{x\in\Omega : \text{\rm dist}(x,\partial\Omega)<\varepsilon\}$
(see Proposition~4.3 \cite{CS11} and references therein).
Moreover, if $\lambda^\star/2<\lambda<\lambda^\star$, then 
$
u_\lambda\geq
u_{\lambda^\star/2} > c\, \text{\rm dist}(\cdot,\partial\Omega)
$
for some positive constant $c$ independent of $\lambda\in (\lambda^\star/2,\lambda^\star)$.
Therefore, letting $t :=c\varepsilon/2$, we have
$\left\{x\in\Omega: u_\lambda(x)\leq t\right\}\subset \Omega_{\varepsilon/2}
\subset \Omega_{\varepsilon}$.

Note that $u_\lambda$ is a solution of the linear equation $-\Delta u_\lambda= 
h(x):= \lambda f(u_\lambda(x))$ in $\Omega_\varepsilon$ and that, by \eqref{newnew},
$u_\lambda$ and the right hand side $h$ are bounded in $L^\infty(\Omega_\varepsilon)$
by a constant independent of~$\lambda$. 
Hence, using interior and boundary estimates for the linear Poisson equation 
and \eqref{last:estimate}, we deduce that 
$$
\left(\int_{M_t}|\nabla v_\lambda|^{\frac{n}{n-2}}\,dV\right)^{2\frac{n-1}{n}}
\leq
C_1
\int_{M_t}|\nabla v_\lambda|^{\frac{n}{n-2}}\,dV
+C_2
$$
for some constants $C_1$ and $C_2$ independent of $\lambda$.

Finally, noting that $2(n-1)/n>1$ (since $n\geq 3$) and $|\nabla u_\lambda|
\leq |\nabla v_\lambda|$ we obtain 
$$
\int_{\{u_\lambda\geq t\}}|\nabla u_\lambda|^{\frac{n}{n-2}+1}\,dx
\leq
\int_{\{u_\lambda\geq t\}}|\nabla v_\lambda|^{\frac{n}{n-2}+1}\,dx
=
\int_{M_t}|\nabla v_\lambda|^{\frac{n}{n-2}}\,dV
\leq C,
$$
for some constant $C$ independent of $\lambda$.
Letting $\lambda$ tend to $\lambda^\star$ in the previous inequality we conclude 
that $u^\star\in W^{1,2\frac{n-1}{n-2}}_0(\Omega)$ proving the theorem.
\end{proof}

\centerline{{\bf Acknowledgments}}

\smallskip
The author was supported by projects MTM2011-27739-C04-01 (Spain) 
and 2009SGR345 (Catalunya).





\begin{thebibliography}{00}

\bibitem{BCMR96}H. Brezis, T. Cazenave, Y. Martel, A. Ramiandrisoa, 
{Blow up for $u_t-\Delta u=g(u)$ revisited},
Adv. Differential Equations {\bf 1} (1996), 73--90.

\bibitem{BV97}
H. Brezis, J.L. V\'azquez, 
{Blow-up solutions of some nonlinear elliptic problems}, 
Rev. Mat. Univ. Compl. Madrid {\bf 10} (1997), 443--469.

\bibitem{Burago} Yu.D. Burago, V.A. Zalgaller, 
Geometric inequalities, Springer Series in Soviet Mathematics. Springer-Verlag, 
Berlin, 1988.

\bibitem{Cabre10} 
X. Cabr\'e, 
{Regularity of minimizers of semilinear elliptic problems up to dimension 4},
Comm. Pure Appl. Math {\bf 63} (2010), 1362--1380.

\bibitem{CS11}
X. Cabr\'e, M. Sanch\'on,
{Geometric-type Sobolev inequalities and applications to the regularity of minimizers},
Preprint: arXiv:1111.2801v1.

\bibitem{CNV}
A. Cesaroni, M. Novaga, M. Valdinoci,
{A symmetry result for the Ornstein-Uhlenbeck operator}.
Preprint: arXiv:1204.0880v1


\bibitem{CR75}
M.G. Crandall, P.H. Rabinowitz,
{Some continuation and variational methods for positive solutions of
nonlinear elliptic eigenvalue problems}, 
Arch. Ration. Mech. Anal. {\bf 58} (1975), 207--218.

\bibitem{Dupaigne} Dupaigne, L.: 
{Stable solutions to elliptic partial differential equations}. 
Monographs and Surveys in Pure and Applied Mathematics, 2011.

\bibitem{FSV08}
A. Farina, B. Sciunzi, E. Valdinoci, 
{Bernstein and De Giorgi type problems: new results via a geometric approach},
Ann. Sc. Norm. Super. Pisa Cl. Sci. (5), 7(4) (2008), 741–-791.

\bibitem{Nedev01}
G. Nedev, 
{Extremal solution of semilinear elliptic equations}, 
Preprint 2001.

\bibitem{SZ1}
P. Sternberg, K. Zumbrun,
{Connectivity of phase boundaries in strictly convex domains},
Arch. Rational Mech. Anal. {\bf 141} (1998), 375--400.

\bibitem{SZ2}
P. Sternberg, K. Zumbrun,
{A {P}oincar\'e inequality with applications to
volume-constrained area-minimizing surfaces},
J. Reine Angew. Math. {\bf 503} (1998), 63--85.	

\end{thebibliography}
\end{document}